\documentclass[11pt,twoside,reqno]{amsart}

\usepackage{color}

\usepackage{microtype}
\usepackage[OT1]{fontenc}
\usepackage{type1cm}
\usepackage{amssymb}

\usepackage[left=2.3cm,top=3cm,right=2.3cm]{geometry}
\numberwithin{equation}{section}

\theoremstyle{plain}
\newtheorem{theorem}{Theorem}[section]

\newtheorem{proposition}[theorem]{Proposition}
\newtheorem{lemma}[theorem]{Lemma}

\theoremstyle{remark}
\newtheorem{remark}[theorem]{Remark}
\newtheorem{example}[theorem]{Example}

\theoremstyle{definition}

\newsavebox{\framedbox}

\renewcommand{\AA}{\mathcal{A}}
\newcommand{\BB}{\mathcal{B}}

\newcommand{\QQ}{\mathcal{Q}}

\newcommand{\R}{\mathbb{R}}

\newcommand{\Q}{\mathbb{Q}}

\newcommand{\N}{\mathbb{N}}

\newcommand{\iii}{\mathtt{i}}
\newcommand{\jjj}{\mathtt{j}}

\newcommand{\ccc}{\mathtt{c}}

\newcommand{\eps}{\varepsilon}

\renewcommand{\atop}[2]{\genfrac{}{}{0pt}{}{#1}{#2}}

\newenvironment{labeledlist}[2][\unskip]
{ \renewcommand{\theenumi}{({#2}\arabic{enumi}{#1})}
  \renewcommand{\labelenumi}{({#2}\arabic{enumi}{#1})}
  \begin{enumerate} }
{ \end{enumerate} }

\DeclareMathOperator{\esssup}{ess\,sup}
\DeclareMathOperator{\essinf}{ess\,inf}

\DeclareMathOperator{\udimloc}{\overline{dim}_{loc}}
\DeclareMathOperator{\ldimloc}{\underline{dim}_{loc}}

\DeclareMathOperator{\dimh}{dim_H}
\DeclareMathOperator{\udimh}{\overline{dim}_H}
\DeclareMathOperator{\ldimh}{\underline{dim}_H}

\DeclareMathOperator{\dimp}{dim_p}
\DeclareMathOperator{\udimp}{\overline{dim}_p}
\DeclareMathOperator{\ldimp}{\underline{dim}_p}

\DeclareMathOperator{\dimb}{dim_B}

\DeclareMathOperator{\diam}{diam}

\DeclareMathOperator{\spt}{spt}

\begin{document}

\title{Local dimensions in Moran constructions}

\author{Antti K\"aenm\"aki}
\address{Department of Mathematics and Statistics \\
         P.O.\ Box 35 (MaD) \\
         FI-40014 University of Jyv\"askyl\"a \\
         Finland}
\email{antti.kaenmaki@jyu.fi}

\author{Bing Li}
\address{Department of Mathematics \\
         South China University of Technology \\
Guangzhou, 510641, P.R.\ China}
\email{scbingli@scut.edu.cn}

\author{Ville Suomala}
\address{Department of Mathematical Sciences \\
         P.O.\ Box 3000 \\
         FI-90014 University of Oulu \\
         Finland}
\email{ville.suomala@oulu.fi}

\thanks{
BL is the corresponding author. AK and BL thank the  mobility grant of the Academy of Finland and the NSFC 11411130372.
BL was supported by the NSFC 11371148 and 11201155, and "Fundamental Research Funds for the Central Universities" SCUT(2015ZZ055). VS acknowledges support from the Academy of Finland. }
\date{\today}

\begin{abstract}
  We study the dimensional properties of Moran sets and Moran measures in doubling metric spaces. In particular, we consider local dimensions and $L^q$-dimensions. We generalize and extend several existing results in this area.
\end{abstract}

\maketitle

\section{Introduction} \label{sec:intro}

The concepts of a Moran set and a Moran construction arise naturally in the study of fractal sets. These concepts extend the model case of self-similar sets and more general graph-directed self-conformal sets by allowing more flexible constructions both in terms of the placement of the basic sets and in that the contraction ratios may vary from step to step. Moran sets have found applications e.g.\ in the study of the structure of the spectrum of quasicrystals and multifractal analysis (see \cite{LiuWen2005}, \cite{FFW2001} and the references therein).

Dimensional properties of various Moran sets and measures have been under extensive study in the real-line and higher dimensional Euclidean spaces; see e.g.\ \cite{FengWenWu1997, HollandZhang2013, HuaRaoWenWu2000, KaenmakiVilppolainen2008, LiWu2011, LiuWen2005, Roychowdhury2010, Roychowdhury2011, WuXiao2011, XueKamae2013} for few aspects of this research line. In non-Euclidean spaces, it is often impossible to construct non-trivial self-similar or self-conformal sets. Nevertheless, in metric spaces that are geometrically doubling, under mild regularity assumptions, it is always possible to define Moran sets by using Moran constructions that are slightly more flexible than the ``standard'' constructions found in the above references. In particular, this means that we must allow fluctuations in the sizes of the basic sets that are ``logarithmically negligible''. Thus, the concepts of Moran sets and measures are natural in these spaces. Apart from \cite{BaloghRohner2007, BaloghTysonWarhurst2009, KaenmakiRajalaSuomala2013, RajalaVilppolainen2011}, where certain classes of self-similar and specific Moran sets and measures are considered, we do not know any attempts to study general Moran constructions in metric spaces.

Our purpose in this note is to generalize some fundamental dimensional results of Moran sets and measures from Euclidean spaces to doubling metric spaces. The paper is organized as follows: In Section \ref{sec:notation}, we set up some notation and recall few auxiliary results. In Section \ref{sec:filtrations} we present a method to calculate local dimensions of measures using a relaxed notion of mesh cubes in the metric setting. The dimensional results for various Moran sets and measures are presented in Section \ref{sec:moran}.

\section{Notation and preliminaries} \label{sec:notation}

We work on a complete metric space $(X,d)$ which is assumed to be \emph{doubling}. This means that there is an integer $N$ (\emph{a doubling constant of} $X$) such
that any closed ball $B(x,r) = \{ y
\in X : d(x,y) \le r \}$ with center $x \in X$ and radius $r>0$ can be
covered by $N$ balls of radius $r/2$.
Notice that, although $x \ne y$ or $r\neq t$, it may happen that
$B(x,r)=B(y,t)$. For simplicity, we henceforth assume that each ball comes with a fixed center and radius.
This makes it possible to use notation such as
$2B=B(x,2r)$ without referring to the centre or radius of the ball $B=B(x,r)$. 

We call any countable collection $\BB$ of pairwise disjoint closed
balls a \emph{packing}. It is called a \emph{packing of $A$} for a
subset $A \subset X$ if the centers of the balls of $\BB$ lie in the
set $A$, and it is a \emph{$\delta$-packing} for $\delta > 0$ if all
of the balls in $\BB$ have radius $\delta$. A $\delta$-packing $\BB$
of $A$ is termed \emph{maximal} if for every $x \in A$ there is $B \in
\BB$ so that $B(x,\delta) \cap B \ne \emptyset$. Note that if $\BB$ is
a maximal $\delta$-packing of $A$, then $2\BB$ covers $A$. Here $2\BB
= \{ 2B : B \in \BB \}$. Observe that doubling
metric spaces are always separable. Hence for each
$\delta>0$ and $A \subset X$ there exists a maximal $\delta$-packing
of $A$.

The doubling property can be stated in several equivalent ways. We recall the following basic covering lemma that is valid for doubling metric spaces \cite{Heinonen2001}:

\begin{lemma} \label{thm:covering_thm}
  For a doubling metric space $X$:
\begin{enumerate}
  \item There are $s>0$ and $c>0$ such that if $0<r<R$ and $\BB$ is an $r$-packing of a closed ball of radius $R$, then the cardinality of $\BB$ is at most $c(r/R)^{-s}$. \label{covering3}
  \item There is $M \in \N$ such that if $A \subset X$ and $\delta>0$, then there are $\delta$-packings of $A$, $\BB_1,\ldots,\BB_M$, whose union covers $A$. \label{covering5}
\end{enumerate}
\end{lemma}

In this article \emph{a measure} exclusively refers to a nontrivial Borel regular (outer) measure
defined on all subsets of $X$ so that bounded sets have finite
measure. The support of a measure $\mu$, denoted by $\spt(\mu)$, is
the smallest closed subset of $X$ with full $\mu$-measure.

The Hausdorff, packing, and box-counting dimensions of a set $E\subset X$ will be denoted by $\dimh(E), \dimp(E)$, and $\dimb(E)$, respectively. We will consider the \emph{upper and lower local dimensions of a measure $\mu$} defined at each $x\in\spt(\mu)$ by
\begin{align*}
  \udimloc(\mu,x) &= \limsup_{r \downarrow 0} \frac{\log\mu(B(x,r))}{\log r}, \\
  \ldimloc(\mu,x) &= \liminf_{r \downarrow 0} \frac{\log\mu(B(x,r))}{\log r},
\end{align*}
respectively. We also define \emph{upper and lower Hausdorff and packing dimensions of $\mu$} by setting
\begin{align*}
  \udimh(\mu) &= \esssup_{x \sim \mu} \ldimloc(\mu,x), \\
  \ldimh(\mu) &= \essinf_{x \sim \mu} \ldimloc(\mu,x), \\
  \udimp(\mu) &= \esssup_{x \sim \mu} \udimloc(\mu,x), \\
  \ldimp(\mu) &= \essinf_{x \sim \mu} \udimloc(\mu,x),
\end{align*}
where $x\sim\mu$ means that $x$ is distributed according to $\mu$.
The reader is referred to \cite{Falconer1997} for the basic properties of Hausdorff and packing, as well as box-counting dimensions. In particular, we recall the following relations between the dimensions of sets and measures supported on them: For a non-empty Borel set $E\subset X$, it is well known that certain metric space versions of Frostman's lemma (see \cite{Howroyd1995, JoycePreiss1995, Falconer1997}) imply that
\begin{align}
\dimh(E)=\sup_{\spt(\mu)\subset E}\udimh(\mu)=\sup_{\spt(\mu)\subset E}\ldimh(\mu)\,,\label{Frostman:Hausdorff}\\
\dimp(E)=\sup_{\spt(\mu)\subset E}\udimp(\mu)=\sup_{\spt(\mu)\subset E}\ldimp(\mu)\,.\label{Frostman:packing}
\end{align}

Let $\mu$ be a measure on
$X$, $q \ge 0$, and $r>0$. For a compact set $A \subset X$, we write
\begin{equation*}
  S_q(\mu,A,\delta) = \sup\Bigl\{ \sum_{B \in \BB} \mu(B)^q :
  \BB \text{ is a $\delta$-packing of } A \cap \spt(\mu) \Bigr\}
\end{equation*}
(with the usual interpretation that $0^q=0$ for $q=0$) and define the \emph{global and local $L^q$-spectrums of $\mu$} to be
\begin{align*}
  \tau_q(\mu) &= \liminf_{\delta \downarrow 0}
  \frac{\log S_q(\mu,X,\delta)}{\log\delta}, \\
  \tau_q(\mu,x) &= \lim_{r \downarrow 0} \liminf_{\delta \downarrow 0}
  \frac{\log S_q(\mu,B(x,r),\delta)}{\log\delta},
\end{align*}
respectively, for all $q \ge 0$.

For $q \ne 1$ we define the \emph{global and local $L^q$-dimensions of $\mu$} by setting
\begin{align*}
  \dim_q(\mu) &= \frac{\tau_q(\mu)}{q-1}, \\
  \dim_q(\mu,x) &= \frac{\tau_q(\mu,x)}{q-1},
\end{align*}
respectively. We note that in order for the global notions to make sense, we assume that $\spt(\mu)$ is bounded.

It is straightforward to see that the mapping $q\mapsto\dim_q(\mu,x)$ is continuous and decreasing on $(0,1)\cup(1,\infty)$ for all $x \in \spt(\mu)$; see \cite[Lemma 2.7]{KaenmakiRajalaSuomala2012}. The interesting behaviour of the mapping is at $1$. This is described in the following proposition which in the present metric setting is proved in \cite[Theorem 3.1]{KaenmakiRajalaSuomala2013}. Euclidean versions of this result can be found in many earlier works, see e.g. \cite{Heurteaux1998}.

\begin{proposition} \label{thm:dim-at-one}
  Suppose that $\mu$ is a measure on a complete doubling metric space $X$. Then
  \begin{equation*}
    \lim_{q \downarrow 1} \dim_q(\mu,x) \le \ldimloc(\mu,x) \le \udimloc(\mu,x) \le \lim_{q \uparrow 1} \dim_q(\mu,x)
  \end{equation*}
  for $\mu$-almost all $x \in X$.
\end{proposition}

\section{Local dimensions via general filtrations}\label{sec:filtrations}

We shall now introduce standing assumptions and notation for a specific collection of Borel sets which we call a general filtration. We assume that there are two decreasing sequences of positive real numbers $(\delta_n)_{n \in \N}$ and $(\gamma_n)_{n \in \N}$ such that
\begin{labeledlist}{F}
  \item $\delta_n \le \gamma_n$ for all $n \in \N$, \label{F-smaller}
  \item $\lim_{n \to \infty} \gamma_n = 0$, \label{F-gammalim}
  \item $\lim_{n \to \infty} \log\delta_n/\log\delta_{n+1} = 1$, \label{F-deltalim}
  \item $\lim_{n \to \infty} \log\gamma_n/\log\delta_n = 1$. \label{F-gammadeltalim}
\end{labeledlist}
For each $n\in\mathbb{N}$, let $\QQ_n$ be a collection of disjoint Borel subsets of $X$ such that each $Q\in\QQ_n$ contains a ball $B_Q$ of radius $\delta_n$, and is contained in a ball of radius $\gamma_n$. We define
$$
  E=\bigcap_{n=1}^\infty\bigcup_{Q\in\QQ_n}Q
$$
and refer the collection $\{ \QQ_n \}_{n \in \N}$ to as the \emph{general filtration of $E$}. For every $x\in E$ and $n\in\N$ there exists a unique $Q\in\QQ_n$ containing $x$, denoted by $Q_n(x)$. Furthermore, if $x \in E$ and $r>0$, then we set $\QQ_n(x,r) = \{ Q \in \QQ_n : Q \cap B(x,r) \ne \emptyset \}$.

Perhaps the most classical example of a general filtration is the dyadic filtration of the Euclidean space. It is worth noting that in doubling metric spaces, it is also possible to define nested partitions quite analogous to the dyadic cubes. A variety of such constructions have appeared in quite different instances. See e.g.\ \cite{Christ1990, AimarBernardisBibiana2007} and the references in \cite{KaenmakiRajalaSuomala2012a}. The  $\delta$-partitions, considered in \cite{KaenmakiRajalaSuomala2013}, are examples of general filtrations which lack the nested structure present in the dyadic constructions.

In the following, we show that local dimensions and $L^q$-dimensions can be equivalently defined via general filtrations. These are variants of \cite[Propositions 3.1 and 3.2]{KaenmakiRajalaSuomala2013} suited for our purposes.

\begin{proposition} \label{partitionlemma}
  Let $X$ be a complete doubling metric space and $\mu$ a non-atomic measure supported on $E$. Then
  \begin{align}
    \udimloc(\mu,x) &= \limsup_{n \to \infty} \frac{\log\mu(Q_n(x))}{\log\delta_n}, \label{dimupart} \\
    \ldimloc(\mu,x) &= \liminf_{n \to \infty} \frac{\log\mu(Q_n(x))}{\log\delta_n}, \label{dimlpart}
  \end{align}
  for $\mu$-almost all $x \in E$.
\end{proposition}

\begin{proof}
  For every $x\in E$ and $n\in \N$ we have $Q_n(x)\subset B(x, 2\gamma_n)$. This, together with \ref{F-gammadeltalim} and \ref{F-deltalim}, implies
  $$
    \udimloc(\mu,x)\le\limsup_{n\to\infty}\frac{\log\mu(B(x, 2\gamma_n))}{\log 2\gamma_{n-1}}\leq\limsup_{n\to\infty}\frac{\log \mu(Q_n(x))}{\log\delta_n}
  $$
  and
  $$
    \ldimloc(\mu,x)\le\liminf_{n\to\infty}\frac{\log\mu(B(x, 2\gamma_n))}{\log 2\gamma_n}\leq\liminf_{n\to\infty}\frac{\log \mu(Q_n(x))}{\log\delta_n}.
  $$
 Therefore, it suffices to show that $\mu(\overline{A} \cup \underline{A}) = 0$, where
  \begin{align*}
    \overline{A} &= \{ x \in X : \udimloc(\mu,x) < \limsup_{n \to \infty} \log\mu(Q_n(x))/\log\delta_n \}, \\
    \underline{A} &= \{ x \in X : \ldimloc(\mu,x) < \liminf_{n \to \infty} \log\mu(Q_n(x))/\log\delta_n \}.
  \end{align*}
  Our plan is to show that
  \begin{equation} \label{eq:local_goal1}
    \overline{A} \cup \underline{A} \subset \bigcup_{\atop{0<s<t<\infty}{s,t \in \Q}} \bigcap_{0<\eps\in\Q} \bigcap_{k=1}^\infty \bigcup_{m=k}^\infty A_m(s,t,\eps),
  \end{equation}
  where
  \begin{align*}
    A_m(s,t,\eps) = \{ x \in X : \;&\mu(B(x,2^{-m(1-\eps)})) > 2^{-ms} \text{ and there is } n \in \N \\ &\text{such that } 2^{-m} \le \delta_n < 2^{-m(1-\eps)} \text{ and } \mu(Q_n(x)) < 2^{-mt(1-\eps)} \},
  \end{align*}
  and then use the Borel-Cantelli Lemma to conclude that $\mu(\overline{A} \cup \underline{A}) = 0$.

  Take $x \in \overline{A}$, choose $s,t \in \Q$ such that $\udimloc(\mu,x) < s < t < \limsup_{n \to \infty} \log\mu(Q_n(x))/\log\delta_n$, and fix $\eps>0$. Choose $r_0>0$ and a sequence $(n_k)_k$ with $n_k \to \infty$ so that
  \begin{equation*}
    \frac{\log\mu(B(x,r))}{\log r} < s < t < \frac{\log\mu(Q_{n_k}(x))}{\log\delta_{n_k}}
  \end{equation*}
  for all $0<r<r_0$ and $k \in \N$. Fix $k$ and choose $m$ such that $2^{-m} \le \delta_{n_k} < 2^{-m+1}$. Since $\delta_n \to 0$ we may assume that $m > \max\{ 1/\eps,k \}$ and all the quantities are smaller than $r_0$. This implies that
  \begin{equation} \label{eq:local1}
    2^{-m} \le \delta_{n_k} < 2^{-m+1}      
    < 2^{-m(1-\eps)}.
  \end{equation}
  Now, choosing $r=2^{-m}$, we have
  \begin{equation*}
    \frac{\log\mu(B(x,2^{-m(1-\eps)}))}{\log 2^{-m}} \le \frac{\log\mu(B(x,2^{-m}))}{\log 2^{-m}} < s
  \end{equation*}
  and hence
  \begin{equation} \label{eq:local2}
    \mu(B(x,2^{-m(1-\eps)})) > 2^{-ms}.
  \end{equation}
  Since
  \begin{equation*}
    t < \frac{\log\mu(Q_{n_k}(x))}{\log\delta_{n_k}} < \frac{\log\mu(Q_{n_k}(x))}{\log 2^{-m(1-\eps)}}
  \end{equation*}
  we have
  \begin{equation} \label{eq:local3}
    \mu(Q_{n_k}(x)) < 2^{-mt(1-\eps)}.
  \end{equation}
  Thus for each $\eps>0$ and $k \in \N$ there is $m \ge k$ so that \eqref{eq:local1}, \eqref{eq:local2}, and \eqref{eq:local3} hold, that is, $x \in A_m(s,t,\eps)$.

  With a same kind of reasoning, we see that for each $x\in \underline{A}$, $k\in\N$ and $\eps>0$, there is $m\ge k$ and $s,t\in\Q$ such that $x\in A_m(s,t,\eps)$.

  It remains to show that $\mu(\overline{A} \cup \underline{A}) = 0$. By the Borel-Cantelli Lemma, it suffices to show that for each $s,t \in \Q$ with $0<s<t$ there is $\eps>0$ so that
  \begin{equation} \label{eq:local_goal2}
    \sum_{m=1}^\infty \mu(A_m(s,t,\eps)) < \infty.
  \end{equation}
  Fix $s,t \in \Q$ with $0<s<t$ and let $h = \log_2 N$, where $N$ is the doubling constant of $X$. Choose $0<\eps<(t-s)/(h+t)$. Fix $m$ and take $x \in A_m(s,t,\eps)$. Let $n$ be as in the definition of $A_m(s,t,\eps)$ and define
  \begin{equation*}
    \AA = \{ Q \in \QQ_n : Q \cap A_m(s,t,\eps) \cap B(x,2^{-m(1-\eps)}) \ne \emptyset \}.
  \end{equation*}
  Since $Q$ is contained in a ball with radius $\gamma_n$ for all $Q \in \QQ_n$, we have $Q \subset B(x,2^{-m(1-\eps)}+2\gamma_n)$ for all $Q \in \AA$. Recalling that $\{ B_Q : Q \in \QQ_n \}$ is a $\delta_n$-packing, Lemma \ref{thm:covering_thm}(\ref{covering3}) (where we can take $s=\log_2 N$) implies
  \begin{equation*}
    \#\AA \le c\biggl( \frac{\delta_n}{2^{-m(1-\eps)}+2\gamma_n} \biggr)^{-h} \le c\biggl( \frac{2^{-m}}{3\times 2^{-m(1-\eps)^2}} \biggr)^{-h}
    \leq c3^h 2^{2mh\eps}\,,
  \end{equation*}
where we have used that $\gamma_n\le\delta_n^{1-\eps}$, which is true for all large $m\in\N$.
  We thus have
  \begin{align*}
    \mu(A_m(s,t,\eps) \cap B(x,2^{-m(1-\eps)})) &\le \sum_{Q \in \AA} \mu(Q) \le c3^h 2^{2mh\eps} 2^{-mt(1-\eps)} \\
    &= c3^h 2^{-ms-m(t-s-\eps(2h+t))} \\
    &\le c3^h 2^{-m(t-s-\eps(2h+t))} \mu(B(x,2^{-m(1-\eps)})).
  \end{align*}
  Using Lemma \ref{thm:covering_thm}(\ref{covering5}), we can cover $A_m(s,t,\eps)$ by a union of at most $M=M(N)$ many $\left(2^{-m(1-\eps)}\right)$-packings of $A_m(s,t,\eps)$ and thus
  \begin{equation*}
    \mu(A_m(s,t,\eps)) \le c4^h M\mu(X) 2^{-m(t-s-\eps(2h+t))}.
  \end{equation*}
  Since $t-s-\eps(2h+t)>0$ by the choice of $\eps$ we have shown \eqref{eq:local_goal2} and thus finished the proof.
\end{proof}

The following proposition shows that local $L^q$-dimensions can as well be defined via general filtrations. We note that a global version of this proposition can be obtained just by removing all the occurences of $x$ and $r$ in the claim.

\begin{proposition} \label{tauq}
  Let $X$ be a complete doubling metric space and $\mu$ a non-atomic measure supported on $E$ and $q \ge 0$. Then
  \begin{equation}\label{tauq:claim}
      \tau_q(\mu,x) = \lim_{r \downarrow 0} \liminf_{n \to \infty} \frac{\log\sum_{Q \in \QQ_n(x,r)} \mu(Q)^q}{\log \delta_n}
  \end{equation}
  for all $x \in \spt(\mu)$.
\end{proposition}

\begin{proof}
To show that the right-hand side of \eqref{tauq:claim} is a lower bound for $\tau_q(\mu,x)$, it suffices to prove that for a fixed $x \in \spt(\mu)$ and $r>0$ we have
\begin{equation}\label{lessequal}
S_q(\mu,B(x,r),\delta)\leq c^{q}\Bigl(\frac{\delta_n}{3\gamma_n}\Bigr)^{-s(q-1)}\Bigl(\frac{\delta_{n+1}}{3\gamma_n}\Bigr)^{-s}\sum_{Q\in\QQ_n(x,r)}\mu(Q)^q.
\end{equation}
and to obtain the estimate in the other direction, it suffices to find a $2\gamma_n$-packing $\mathcal{B}$ of $B(x,r)\cap\spt(\mu)$ such that
\begin{equation}\label{largerestimate}
\sum_{Q\in\QQ_n(x,r)}\mu(Q)^q\leq c\Bigl(\frac{\delta_n}{3\gamma_n}\Bigr)^{-s}\sum_{B\in\mathcal{B}}\mu(B)^q.
\end{equation}
Both of these claims can be verified by following closely the proof of \cite[Proposition 3.2]{KaenmakiRajalaSuomala2013}. We omit the details.
\end{proof}

\section{Dimensions in Moran constructions}\label{sec:moran}

In the definition of the Moran construction, we use a suitable notion of codetrees. Let $I$ be a countable (finite or infinite) set, $I^\infty = I^\N$, and $I^0 = \{ \varnothing \}$. We refer to the elements of $I^\infty \cup \bigcup_{j=0}^\infty I^j$ as \emph{words}. The length of a finite word $\iii = i_1 i_2 \cdots i_m$ is denoted by $|\iii|=m$. We set $\iii|_n = i_1 \cdots i_n$ for all $1 \le n \le |\iii|$ and $\iii|_0 = \varnothing$. The concatenation of a finite word $\iii$ and another word $\jjj$ is denoted by $\iii\jjj$. We also set $\iii^- = \iii|_{n-1}$ for all $\iii \in I^n$ and $n \in \N$.

If $\Sigma \subset \bigcup_{n=0}^\infty I^n$, then we set $\Sigma_n = \Sigma \cap I^n$ for all $n \ge 0$ and $\Sigma_\infty = \{ \iii \in I^\infty : \iii|_n \in \Sigma_n \text{ for all } n \in \N \}$. If $\iii \in \Sigma$ and $\jjj = \iii i \in \Sigma_{|\iii|+1}$ for some $i \in I$, then we say that $\jjj$ is an \emph{offspring} of $\iii$ and write $\jjj \prec \iii$. We say that $\Sigma$ is a \emph{codetree}, if $\varnothing \in \Sigma$, $\iii^{-}\in\Sigma$ whenever $\varnothing\neq\iii\in\Sigma$, and the number of offsprings $\#\{\jjj \in \Sigma_{|\iii|+1} : \jjj \prec \iii \}$ is positive and finite for all $\iii \in \Sigma$. If $\Sigma$ is a codetree and $X$ is a complete doubling metric space, then we say that a collection $\{E_\iii \subset X : \iii \in \Sigma\}$ of compact sets with positive diameter is a \emph{Moran construction}, if
\begin{labeledlist}{M}
  \item $E_\iii \subset E_{\iii^-}$ for all $\varnothing\neq\iii \in \Sigma$, \label{M-nested}
  \item $\lim_{n\rightarrow\infty}\diam(E_{\iii|_n})=0$
  for each $\iii\in\Sigma_\infty$,
  \label{M-diam}
  \item $E_{\iii i}\cap E_{\iii j}=\emptyset$ if $\iii,\iii i, \iii j\in\Sigma$ and $i\neq j$, \label{M-disjoint}
  \item for each $\iii\in\Sigma$, there is $x\in E_\iii$ such that $B\left(x,C_0 \diam(E_\iii)\right)\subset E_\iii$, \label{M-inradius}
  \item $\lim_{n \to \infty} \log\diam(E_{\iii|_n})/\log\min\{ \diam(E_{\jjj}) : \jjj \prec \iii|_n \}=1$ uniformly
  for all $\iii \in \Sigma_\infty$. \label{M-logdiam}
\end{labeledlist}
In \ref{M-inradius}, $C_0>0$ is a uniform constant. Note that if $\jjj \prec \iii$, then $E_\jjj \subset E_\iii$ and $|\jjj| = |\iii|+1$.

We define the \emph{limit set} of the Moran construction to be the compact set $E = \bigcap_{n \in \N} \bigcup_{\iii \in \Sigma_n} E_\iii$.
If $\iii \in \Sigma_\infty$, then the single point in the set $\bigcap_{n \in \N} E_{\iii|_n} \subset E$ is denoted by $x_\iii$. Observe that the mapping $\iii \mapsto x_\iii$, which we denote by $\Pi$, is a bijective projection from $\Sigma_\infty$ to $E$.
The main goal of this article is to obtain information about the dimensional properties of $E$ and measures supported by $E$. Since $E$ is bounded and in doubling metric spaces, bounded sets are totally bounded, we may, without loss of generality, assume that $X$ is a compact doubling metric space.

Given a codetree $\Sigma$, a Moran construction $\{E_\iii\}_{\iii\in \Sigma}$, and a measure $\mu$ on $E$, we define
\begin{align*}
  p_\iii &= \mu(E_\iii)/\mu(E_{\iii^{-}}), \\
  c_\iii &= \diam(E_\iii)/\diam(E_{\iii^{-}})
\end{align*}
for all $\varnothing \ne \iii \in \Sigma$ and $p_\varnothing = \mu(E_\varnothing)$, $c_\varnothing = \diam(E_\varnothing)$. If $\mu(E_\iii) = 0$ for some $\iii$, then we interpret that $p_\jjj = 0$ for all $\jjj \prec \iii$.
The numbers $p_\iii$ and $c_\iii$ can be thought as functions from $\Sigma$ to $[0,1]$. Observe that $\prod_{n=0}^{|\iii|} p_{\iii|_n} = \mu(E_\iii)$ and $\sum_{\jjj \prec \iii} p_\jjj = 1$ for all $\iii \in \Sigma$ (as long as $\mu(E_\iii)>0$). For convenience of notation, we will also denote
\begin{equation}\label{ccc:def}
\ccc_\iii=\prod_{i=1}^n c_\iii,
\end{equation}
for $\iii=i_1\ldots i_n\in\Sigma_n$.

We say that the Moran construction $\{E_\iii\}_{\iii\in\Sigma}$ is \emph{spatially symmetric}, if $I=\N$ and for each $k \in \N$ there is $N_k \in \N$ so that $\Sigma_n = \{ i_1 i_2 \cdots i_n : i_k \in \{1,2,\ldots, N_k\} \text{ for all } k \in \{1,2,\ldots, n\} \}$ for all $n\in\N$. We also assume that for every $k \in \N$ and $i \in \{1,2,\ldots,N_k\}$ there is $0 < c_{k,i} < 1$ so that
\begin{equation}\label{def:ss}
c_{\iii i} = c_{k,i}\text{ for all }\iii \in \Sigma_k.
\end{equation}
Although being useful in the Euclidean spaces, in general doubling metric spaces this notion does not always make sense. For instance, it is often impossible to find any such constructions. Furthermore, the property of being spatially homogeneous fails to be bi-Lipschitz invariant. For these reasons we introduce the following relaxed notion: We say that the Moran construction $\{E_\iii\}_{\iii\in\Sigma}$ is \emph{asymptotically spatially symmetric} if in the definition, the assumption \eqref{def:ss}, is replaced by the weaker condition
\begin{equation}\label{def:ass}
\lim_{n\to\infty}\frac{\sum_{j=1}^n \log c_{\iii|_j}}{\sum_{j=1}^n \log c_{j, i_j}}=1\text{ for all }\iii\in\Sigma_\infty.
\end{equation}
This notion is more useful in the sense that it is always possible to find such constructions under very general assumptions. For instance, in addition to the doubling property, it suffices to assume that $X$ is uniformly perfect (see \cite[Remark 4.3]{KaenmakiRajalaSuomala2013}). Recall that $X$ is \emph{uniformly perfect}, if there is $\eta>0$ such that
\begin{equation*}
B(x,r)\setminus B(x,\eta r)\neq\emptyset\text{ whenever }X\setminus B(x,r)\neq\emptyset.
\end{equation*}
We illustrate this in the light of the following simple example. More complicated Moran sets can be obtained with the same method, and it is clear from the construction that also the uniform perfectness assumption could be relaxed to some extent.

\begin{example}
Let $X$ be uniformly perfect with parameter $\eta$. We construct an asymptotically spatially symmetric Moran set $E$ with $\Sigma_\infty=\{0,1\}^\N$ and $c_{k,i}=\eta^2/3=:c$ for all $k\in\N$, $i\in\{0,1\}$. Without loss of generality, we assume that $\diam(X)>2$ so that we may find $x_\varnothing\in X$ with $X\setminus B(x_\varnothing,1)\neq\varnothing$. Set $r_\varnothing=1$ and $E_\varnothing=B(x_\varnothing,r_\varnothing)$. If $E_\iii=B(x_\iii,r_\iii)$, $\iii\in\{0,1\}^{n-1}$, we construct $E_{\iii 0},E_{\iii 1}$ inductively as follows:
Set $x_{\iii 0}=x_\iii$. Using the uniform perfectness, pick $x_{\iii 1}\in B(x_{\iii 0},r_\iii)\setminus B(x_{\iii 0}, \tfrac23\eta r_\iii)$ such that $B(x_{\iii 1},\tfrac{\eta}{3}r_\iii)\subset B(x_\iii, r_\iii)$. For $i=0,1$, define $E_{\iii i}=B(x_{\iii i}, r_{\iii i})$, where
\[r_{\iii i}=\inf\{r>0\,:\,\diam B(x_{\iii i},r)\ge c^n\}\,.\]

Now, it follows immediately that $\diam E_\iii\ge c^{n}$ for all $\iii\in\{0,1\}^n$. Further, using induction and the uniform perfectness, we have $\diam(E_\iii)\le 2 r_\iii< 2\eta^{-1}c^{n}$. Finally, since $\diam B(x_{\iii i},\tfrac{\eta}{3}r_\iii)\ge c^n$ if $i\in\{0,1\}^{n-1}$, again using uniform perfectness and the definition of $c$, it follows that $E_{\iii 0}, E_{\iii 1}$ are pairwise disjoint compact subsets of $E_\iii$. It is now clear that the Moran construction $\{E_\iii\}$ is asymptotically spatially symmetric with $c_{k,i}\equiv c$. Indeed,
\[
\frac{\sum_{j=1}^n \log c_{\iii|_j}}{\sum_{j=1}^n \log c_{j, i_j}}
=\frac{\log\left(\diam E_{\iii|_n}/\diam E_{i_1}\right)}{\log c^n}\,,\]
and we have $A^{-1} c^n\le \diam E_{\iii|_n}/\diam E_\varnothing\le A c^n$ for some absolute constant $A<\infty$.
\end{example}

We next transfer Proposition \ref{partitionlemma} into the setting of Moran constructions.

\begin{lemma} \label{thm:moran_dim}
  Suppose that $X$ is a complete doubling metric space, $\{ E_\iii \}_{\iii \in \Sigma}$ is a Moran construction, and $\mu$ is a measure supported on $E$. Then
  \begin{align*}
    \udimloc(\mu,x_\iii) &= \limsup_{n \to \infty} \frac{\log\mu(E_{\iii|_n})}{\log\diam(E_{\iii|_n})}, \\
    \ldimloc(\mu,x_\iii) &= \liminf_{n \to \infty} \frac{\log\mu(E_{\iii|_n})}{\log\diam(E_{\iii|_n})},
  \end{align*}
  for $\mu$-almost all $x_\iii \in E$.
\end{lemma}

\begin{proof}
  Without loss of generality, suppose that $\diam(E_\varnothing)\le 1$. Define
  $$\gamma_n=\min\{\diam(E_\iii): \iii \in \Sigma_n\}$$
  and
  $$\QQ_n=\{E_\jjj: \diam(E_\jjj)\leq\gamma_n<\diam(E_{\jjj^{-}})\}$$
  for all $n \in \N$. By the assumption \ref{M-logdiam}, for any $k\in \N$,  there exists $N_k\in\mathbb{N}$ such that for each $n\geq N_k$ the inequality
  $$\diam(E_{\iii|_n})\leq \min_{\jjj\prec\iii|_n}\diam(E_\jjj)^{1-1/k}$$
  holds for all $\iii\in\Sigma_\infty$. We can choose the sequence $(N_k)_{k \in \N}$ such that it is strictly increasing. Let $0<C_0<1$ be as in \ref{M-inradius} and define
  $$\delta_n=C_0\gamma_n^{1/(1-1/k)}$$
  for all $N_k\leq n<N_{k+1}$ and $k \in \N$.
  Therefore we have
  $$\lim_{n\to\infty}\frac{\log\gamma_n}{\log\delta_n}=\lim_{k\to\infty}(1-\tfrac{1}{k})=1.$$
  It is clear that $\delta_n\leq \gamma_n$ and it follows from \ref{M-diam} that $\lim_{n\to\infty}\gamma_n=0$. We have now verified that the sequences $(\gamma_n)_{n \in \N}$ and $(\delta_n)_{n \in \N}$ satisfy \ref{F-smaller}, \ref{F-gammalim}, and \ref{F-gammadeltalim}. Let $k \in \N$, $n \ge N_k$, and choose $\iii \in \Sigma_{n+1}$ so that $\gamma_{n+1} = \diam(E_\iii)$. Then $\gamma_n \le \diam(E_{\iii^-}) \le \gamma_{n+1}^{1-1/k}$ and
  \begin{equation*}
    \frac{\log\gamma_n}{\log\gamma_{n+1}} \ge 1-\tfrac{1}{k}
  \end{equation*}
  for all $n \ge N_k$. This, and the fact that we may switch $\gamma_n$ to $\delta_n$ in the limit, shows that \ref{F-deltalim} holds. Furthermore, for any $Q\in\QQ_n$, the choices of $\gamma_n$ and $\delta_n$ together with \ref{M-inradius} imply that $Q$ contains a ball of radius $\delta_n$. Therefore the collection $\{ \QQ_n \}_{n \in \N}$ is a general filtration for $E$ and the claims follow from Proposition \ref{partitionlemma}.
\end{proof}

It has been observed recently in many occasions that the local dimensions can be expressed in terms of certain local entropy averages; see e.g.\ \cite{LlorenteNicolau2004, HochmanShmerkin2012, Shmerkin2011}.
The following is a variant of this result, where instead of $\R^d$ we consider a doubling metric space, and the dyadic (or $m$-adic) cubes are replaced by the basic sets $E_\iii$.
As a particular case, we recover and extend the recent results of Lou and Wu \cite{LouWu2010} and Li and Wu \cite{LiWu2011} for the local dimensions of spatially symmetric Moran measures in the Euclidean setting. For instance, in \cite{LouWu2010,LiWu2011}, it is assumed that the ``contractions'' $c_\iii$ are uniformly bounded away from zero.

\begin{theorem}\label{entropyaverages}
Suppose that $X$ is a complete doubling metric space, $\{ E_\iii \}_{\iii \in \Sigma}$ is a Moran construction, and $\mu$ is a measure on $E$. If
\begin{equation}\label{diamspeed}
\liminf_{n\to\infty}-\frac{1}{n}\log\diam(E_{\iii|_n})>0\text{ for each }\iii\in\Sigma_\infty
\end{equation}
 and
\begin{equation} \label{L2bound}
  \sum_{n=1}^\infty n^{-2} \sup_{\iii\in\Sigma_n} \sum_{\jjj\prec\iii}p_{\jjj}\left((\log p_\jjj)^2+(\log c_\iii)^2\right) < \infty,
\end{equation}
then
\begin{align*}
  \udimloc(\mu,x_\iii) &= \limsup_{N \to \infty} \frac{\sum_{n=1}^N \sum_{\jjj \prec \iii|_n} p_\jjj \log p_\jjj}{\sum_{n=1}^N \sum_{\jjj \prec \iii|_n} p_\jjj \log c_\jjj}, \\
  \ldimloc(\mu,x_\iii) &= \liminf_{N \to \infty} \frac{\sum_{n=1}^N \sum_{\jjj \prec \iii|_n} p_\jjj \log p_\jjj}{\sum_{n=1}^N \sum_{\jjj \prec \iii|_n} p_\jjj \log c_\jjj},
\end{align*}
for $\mu$-almost all $x_\iii \in E$.
\end{theorem}

\begin{proof}
  Denote by $\nu$ the measure on $\Sigma_\infty$ satisfying $\nu([\iii]) = \mu(E_\iii)$ for all $\iii \in \Sigma$. Let $\mathcal{F}_n$ be the $\sigma$-algebra generated by $\Sigma_n$. This is an increasing filtration that converges to the Borel $\sigma$-algebra of $\Sigma_\infty$. Define $I_n(\iii) = \log p_{\iii|_n}$ and $H_n(\iii) = \sum_{\jjj \prec \iii|_n} p_\jjj\log p_\jjj$ for all $\iii \in \Sigma_\infty$. Then $\mathbb{E}(I_n\,|\,\mathcal{F}_{n-1})=H_{n-1}$. Moreover,
  \[
    \sum_{n=1}^\infty n^{-2} \mathbb{E}\bigl( (I_n-H_{n-1})^2 \bigr) < \infty
  \]
  by the assumption \eqref{L2bound}. Thus, the law of large numbers for martingale differences (see \cite[Theorem 3 in Chapter VII.9]{Feller1971}) implies that
  \begin{align*}
    \lim_{N \to \infty} \frac1N \Bigl( \log\mu(E_{\iii|_N}) - \sum_{n=1}^N \sum_{\jjj \prec \iii|_n} p_\jjj\log p_\jjj \Bigr)
    &= \lim_{N \to \infty} \frac1N \Bigl( \sum_{n=1}^N \log p_{\iii|_n} - \sum_{n=0}^{N-1} \sum_{\jjj \prec \iii|_n} p_\jjj\log p_\jjj \Bigr) \\
    &= \lim_{N \to \infty} \frac1N \sum_{n=1}^N \bigl( I_n(\iii) - H_{n-1}(\iii) \bigr) = 0
  \end{align*}
  for $\nu$-almost all $\iii \in \Sigma_\infty$ (and thus for $\mu$-almost all $x_\iii\in E$).
  Similarly, defining $D_n(\iii) = \log c_{\iii|_n}$ and $J_n(\iii) = \sum_{\jjj \prec \iii|_n} p_\jjj\log c_\iii$, we have
  \[
    \lim_{N \to \infty} \frac1N \Bigl( \log\diam(E_{\iii|_N}) - \sum_{n=1}^N \sum_{\jjj \prec \iii|_n} p_\jjj\log c_\jjj \Bigr) = \lim_{N \to \infty} \frac1N \sum_{n=1}^N \bigl( D_n(\iii) - J_n(\iii) \bigr)=0
  \]
  for $\mu$-almost all $x_\iii \in E$.
  Thus, writing
  \begin{equation*}
    \frac{\log\mu(E_{\iii|_N})}{\log\diam(E_{\iii|_N})} = \frac{\sum_{n=1}^N \sum_{\jjj \prec \iii|_n} p_\jjj\log p_\jjj - \bigl( \sum_{n=1}^N \sum_{\jjj \prec \iii|_n} p_\jjj\log p_\jjj - \log\mu(E_{\iii|_N}) \bigr)}{\sum_{n=1}^N \sum_{\jjj \prec \iii|_n} p_\jjj\log c_\jjj - \bigl(\sum_{n=1}^N \sum_{\jjj \prec \iii|_n} p_\jjj\log c_\jjj - \log\diam(E_{\iii|_N}) \bigr)},
  \end{equation*}
and taking \eqref{diamspeed} into account, we observe that
  \[
    \lim_{N \to \infty} \biggl( \frac{\log\mu(E_{\iii|_N})}{\log\diam(E_{\iii|_N})} - \frac{\sum_{n=1}^N \sum_{\jjj \prec \iii|_n} p_\jjj\log p_\jjj}{\sum_{n=1}^N \sum_{\jjj \prec \iii|_n} p_\jjj\log c_\jjj} \biggr) = 0
  \]
  for $\mu$-almost all $x_\iii \in E$.
  The claim now follows from Lemma \ref{thm:moran_dim}.
\end{proof}

We do not know how sharp the above result is and would it be possible to remove e.g.\ the assumption \eqref{L2bound}. The following example shows why the above method fails if \eqref{L2bound} is violated.

\begin{example}
Let $I=\{0,1\}$, $\Sigma=\bigcup_{n\in\N}I^n$ and set $p_{\iii}=\tfrac12$ for all $\iii\in\Sigma$ and further for each $\iii\in I^n$, let $c_{\iii0}=e^{-1}$, $c_{\iii1}=e^{-n}$. Then, in the notation of the proof of Theorem \ref{entropyaverages}, $|D_N(\iii)-J_N(\iii)|$ is of order $N$ for each $N$ so that $N^{-1}\sum_{n=1}^N(D_n(\iii)-J_n(\iii))$ fails to converge for all $\iii\in\Sigma_\infty$.
\end{example}

We will now turn to the dimension of the limit set $E$. In Euclidean spaces, the dimensions of spatially symmetric Moran constructions have been studied e.g.\ in \cite{FengWenWu1997, HollandZhang2013, WangWu2008, WenWu2005}. In the metric setting, the dimensional properties of a special class of ``asymptotically self similar'' Moran constructions were studied in \cite{KaenmakiRajalaSuomala2012}. We shall give a result for asymptotically homogeneous Moran sets in complete doubling metric spaces. We first present a lemma that reduces the amount of technicalities by allowing us to consider the spatially symmetric case only.

\begin{lemma}\label{lem:reduction}
Let $\{E_\iii\}_{\iii\in\Sigma}$ be an asymptotically spatially symmetric Moran construction with the defining weights $c_{k,i}$, $k\in\N$, $i\in \{1,2,\cdots, N_k\}$. Define a metric $\rho$ on $\Sigma_\infty$ by $\rho(\iii,\jjj)=\prod_{k=1}^n c_{k,i_k}$, provided $\iii\wedge\jjj=i_1\cdots i_n$, $\rho(\iii,\iii)=0$ and $\rho(\iii,\jjj)=1$ if $\iii\wedge\jjj=\varnothing$. Equipping $\Sigma_\infty$ with this metric, we have
\[\dimh(\Sigma_\infty)=\dimh(E)\quad\text{and}\quad\dimp(\Sigma_\infty)=\dimp(E).\]
Furthermore, $\diam_\rho([i_1\cdots i_n])=c_{1,i_1}\times\cdots\times c_{n,i_n}$ for all $i_1\cdots i_n\in\Sigma_n$. In particular, $\Sigma_\infty$ is spatially symmetric with the defining weights $c_{k,i}$.
\end{lemma}

\begin{proof}
The statement about the diameters and that $\Sigma_\infty$ is spatially symmetric with the defining weights $c_{k,i}$ is clear from the construction. Let $\nu$ be a measure on $\Sigma_\infty$ and set $\mu=\Pi\nu$. From Lemma \ref{thm:moran_dim} and \eqref{def:ass} we infer that for $\nu$-almost all $\iii\in\Sigma$ (and thus for $\mu$-almost all $x_\iii\in E$), we have
\[\ldimloc(\nu,\iii)=\liminf_{n \to \infty} \frac{\log\nu([\iii|_n])}{\sum_{k=0}^n\log c_{k,i_k}}=\liminf_{n \to \infty} \frac{\log\mu(E_{\iii|_n})}{\log\diam(E_{\iii|_n})}=\ldimloc(\mu, x_\iii)\,.\]
In particular, $\dimh(\nu)=\dimh(\mu)$. Similarly $\dimp(\nu)=\dimp(\mu)$. The claims now follow by recalling \eqref{Frostman:Hausdorff} and \eqref{Frostman:packing}.
\end{proof}

The following theorem gives a formula for the dimension of the limit set of a spatially symmetric Moran construction. 

\begin{theorem}\label{Morandim}
Suppose that $X$ is a complete doubling metric space and $\{E_\iii\}_{\iii \in \Sigma}$ is an asymptotically spatially symmetric Moran construction.
Then $\dimh(E)=s_*$ and $\dimp(E)=\dimb(E)=s^*$ where $s_*$ and $s^*$ are the lower and upper limits, respectively, of the sequence $(s_n)$
defined by the formula
\[\prod_{k=1}^n\sum_{i=1}^{N_k}c_{k,i}^{s_n}=1.\]
\end{theorem}

This theorem extends the results of Hua, Rao, Wen, and Wu \cite{HuaRaoWenWu2000} in several different ways. Most importantly, instead of a Euclidean space, the result applies in any complete doubling space, and the notion of spatial symmetry is also relaxed. In the proof, we will make use of the similar ideas as in \cite{HuaRaoWenWu2000}. In particular, we apply the following lemma. Recall the notation $\ccc_\iii$ from \eqref{ccc:def}.

\begin{lemma}\label{lemma:Huaetal}
\begin{enumerate}
\item\label{1} Suppose that in the setting of Theorem \ref{Morandim}, $\{E_\iii\}$ is spatially symmetric. Let $s>0$. Suppose that $\Xi$ is a finite collection of words $\iii\in\bigcup_{k_1\le k\le k_2}\Sigma_k$ with $E\subset\bigcup_{\iii\in\Xi}E_\iii$. Then there is $k_1\le k\le k_2$ such that $\sum_{\iii\in\Sigma_k}\ccc_{\iii}^s\le\sum_{\iii\in\Xi}\ccc_{\iii}^s$.
\item\label{2} If in addition, the elements $E_{\iii}$, $\iii\in\Xi$, are pairwise disjoint, then there is $k_1\le k\le k_2$ such that $\sum_{\iii\in\Sigma_k}\ccc_{\iii}^s\ge\sum_{\iii\in\Xi}\ccc_{\iii}^s$.
\end{enumerate}
\end{lemma}

\begin{proof}
The proof is identical to that of Lemma 2.1 and Remark 2.1 in \cite{HuaRaoWenWu2000}.
\end{proof}

\begin{proof}[Proof of Theorem \ref{Morandim}]
Using Lemma \ref{lem:reduction}, we may assume that the Moran construction is spatially symmetric and that $\diam(E_{\iii})=\ccc_\iii$ for all $\iii\in\Sigma_*$. The estimate $\dimh(E)\le s_*$ is now immediate from the definition of $s_*$.

To prove that $\dimh(E)\ge s_*$, fix $\varepsilon>0$. Let $s<s_*-\varepsilon$ and $\delta>0$. Let $\mathcal{B}=\{B_1,B_2,\ldots\}$ be a covering of $E$ with balls of radius at most $\delta$. For each $B=B_i$, denote
\[\Upsilon(B)=\left\{\iii\in\Sigma : E_\iii\cap B\neq\emptyset \text{ and } \ccc_\iii\le \diam(B)<\ccc_{\iii^{-}}\right\}.\]
From \ref{M-logdiam}, it follows that for any $c>0$, if $\delta>0$ is small enough (depending on $c,\varepsilon$), then
$\diam(E_\iii)\ge\diam(B)^{1+c\varepsilon}$ for all $\iii\in\Upsilon(B)$.
Combining this with Lemma \ref{thm:covering_thm}, and choosing $c>0$ small enough it follows that
\[\#\Upsilon(B)\le\diam(E_\iii)^{-\varepsilon}\,,\]
for all $B\in\mathcal{B}$ and any $\iii\in\Upsilon(B)$.
Thus, we will have
\begin{equation}\label{seps}
\begin{split}
\sum_{B\in\mathcal{B}}\diam(B)^s&\ge \sum_{i=1}^\infty\max\{\diam(E_\iii)^{\varepsilon}\,:\,\iii\in\Upsilon(B_i)\}\sum_{\iii\in\Upsilon(B_i)}\diam(E_{\iii})^s\\
&\ge\sum_{B\in\mathcal{B}}\sum_{\iii\in\Upsilon(B)}\diam(E_\iii)^{s+\varepsilon}
\ge\sum_{\iii\in\Sigma_k}\ccc_{\iii}^{s+\varepsilon}\,,
\end{split}
\end{equation}
where the last estimate follows from Lemma \ref{lemma:Huaetal}\eqref{1} for some
$k\ge\min\{|\iii| : \iii\in\Upsilon(B) \text{ and } B\in\mathcal{B}\}$. Note that by compactness, we only need to consider the case when $\mathcal{B}$ is finite so that Lemma \ref{lemma:Huaetal} applies. Since $k\rightarrow\infty$ as $\delta\downarrow 0$, we observe that the right-hand side of \eqref{seps} is $\ge 1$ when $\delta>0$ is small enough. This implies $\mathcal{H}^s(E)>0$ and thus, letting $\varepsilon\downarrow 0$, we get $\dim_H(E)\ge s_*$.

Next, let us turn to the claims on box and packing dimension. Suppose that $\delta>0$ and let $\mathcal{B}$ be a $\delta$-packing of $E$. For each $n\in\N$, let
\[\Psi(n)=\{\iii\in\Sigma_n\,:\,\ccc_{\iii}<\delta\le\ccc_{\iii^{-}}\text{ and }E_\iii\subset B\text{ for some }B\in\mathcal{B}\}\,.\]
 From \ref{M-logdiam} it follows that for all large $k_1\in\N$ (depending on $\varepsilon$ only) such that $\delta^{s^*+2\varepsilon}\le\ccc_{\iii}^{s^*+\varepsilon}$ for each $\iii\in\Sigma_n$, $n\ge k_1$. Further, if $\delta$ is small enough (depending on $k_1$), then $\Psi(n)=\varnothing$ for $n< k_1$. Also, by \ref{M-diam}, there is $k_2\in\N$ (depending on $\delta$) such that $\Psi(n)=\varnothing$ for all $n>k_2$.  Thus
\begin{align}\label{sepslb}
 \#\mathcal{B}\delta^{s^*+2\varepsilon}\le\sum_{n=k_1}^{k_2}\sum_{\iii\in\Psi(n)}\ccc_\iii^{s^*+\varepsilon}
 \le\sum_{\iii\in\Sigma_k}\ccc_{\iii}^{s^*+\varepsilon}\,,
\end{align}
where the last estimate follows from Lemma \ref{lemma:Huaetal}\eqref{2} for some $k_1\le k\le k_2$. But for $k_1$ large enough (i.e.\ for $\delta$ small enough), the upper bound in \eqref{sepslb} is at most $1$ and this leads to the estimate $\dimb(E)\le s^*+\varepsilon$. Hence, letting $\varepsilon\downarrow 0$, we get $\dim_B(E)\le s^*$.

Finally, using \ref{M-inradius}, for each $n$, we may construct balls $B_\iii\subset E_\iii$, $\iii\in\Sigma_n$, such that
\[\sum_{\iii\in\Sigma_n}\diam(B_\iii)^{s_n}\ge C_{0}^{-1}\sum_{\iii\in\Sigma_n}\ccc_{\iii}^{s_n}=C_0^{-1}\,.\]
Letting $n\rightarrow\infty$, this implies $\mathcal{P}_{0}^s(E)>0$ for each $s<s^*$, where $\mathcal{P}_0^s$ is the packing premeasure (see \cite{Falconer1997}). Using the definition of spatial symmetry, this readily generalizes to $\mathcal{P}_{0}^s(E\cap U)>0$ for any open set $U\subset X$ with $U\cap E\neq\emptyset$ and hence $\dimp(E)\ge s_0$ (see \cite[Corollary 3.9]{Falconer1990}).
\end{proof}

We say that a Moran construction is \emph{(asymptotically) homogeneous}, if it is (asymptotically) spatially symmetric and for each $k \in \N$ there exists $0 < c_k < 1$ so that $c_{k,i} = c_k$ for all $i \in \{ 1,\ldots,N_k \}$.

\begin{remark}\label{rm}
(1) If $E$ in Theorem \ref{thm:moran_dim} is asymptotically homogeneous, then the formulas for $\dimh(E)$ and $\dimp(E)$ simplify to
\begin{align*}
\dimh(E) &= \liminf_{n\rightarrow\infty}\frac{\sum_{k=1}^n\log N_k}{-\sum_{k=1}^n\log c_k}, \\
\dimp(E) &= \limsup_{n\rightarrow\infty}\frac{\sum_{k=1}^n\log N_k}{-\sum_{k=1}^n\log c_k}.
\end{align*}
This generalizes the results of \cite{FengRaoWu1997, Hua1994, HuaLi1996}.

(2) For self-similar and related sets, the simplest way to obtain the dimension of the set is to consider the local dimensions of a ``natural measure'' on the set and then use \eqref{Frostman:Hausdorff}, \eqref{Frostman:packing}. One might hope to use Theorem \ref{entropyaverages} to give such a ``measure proof'' for Theorem \ref{thm:moran_dim}. However, in the present setting this does not seem appropriate since the ``contractions'' $\{c_{k,i}\}_i$ may vary freely as $k$ changes, and there is no natural Gibbs measure that would catch the correct dimension for $E$ (see also \cite{WuXiao2011}).
\end{remark}

We say that a measure $\mu$ on the limit set of an asymptotically homogeneous Moran construction is \emph{uniformly distributed} if
\begin{equation*}
  \mu(E_\iii) = \prod_{k=1}^{n} N_k^{-1}
\end{equation*}
for all $\iii \in \Sigma_n$ and $n \in \N$.

We conclude with the following indication of monofractality for the uniformly distributed measures on homogeneous Moran sets.

\begin{proposition}\label{dimqlowerupper}
  Suppose that $X$ is a complete doubling metric space, $\{ E_\iii \}_{\iii \in \Sigma}$ is an asymptotically homogeneous Moran construction and $\mu$ is a uniformly distributed measure supported on $E$.
  Then
  \begin{equation} \label{qge1}
    \dim_q(\mu,x)=\dim_q(\mu) = \liminf_{n \to \infty} \frac{\sum_{k=1}^{n} \log N_k}{- \sum_{k=1}^n \log c_k} = s_*
  \end{equation}
  for all $x \in E$ and $q>1$, and
  \begin{equation} \label{qle1}
    \dim_q(\mu,x)=\dim_q(\mu) = \limsup_{n \to \infty} \frac{\sum_{k=1}^n \log N_k}{- \sum_{k=1}^{n} \log c_k} = s^*
  \end{equation}
  for all $x \in E$ and $0<q<1$.
\end{proposition}

\begin{proof}
As in the proof of Lemma \ref{thm:moran_dim}, defining $\gamma_n=\min\{\diam(E_\jjj)\,:\,\jjj\in\Sigma_n\}$ and
$\QQ_n=\{E_\iii: \iii\in\Sigma_n\}$ guarantees the existence of a sequence $(\delta_n)_{n \in \N}$ such that the sequences $(\gamma_n)_{n \in \N}$ and $(\delta_n)_{n \in \N}$ satisfy \ref{F-smaller}--\ref{F-gammadeltalim}, and the collection $\{ \QQ_n \}_{n \in \N}$ is a general filtration for $E$. Furthermore, from the asymptotic homogeneity, it follows that
\[\lim_{n\rightarrow\infty}
\frac{\log\gamma_n}{\sum_{j=1}^n\log c_j}=1.
\]

To show \eqref{qge1}, we use Proposition \ref{tauq} and \ref{F-gammadeltalim} to calculate
\begin{align*}
  \tau_q(\mu,x)&=\lim_{r \downarrow 0} \liminf_{n \to \infty} \frac{\log\sum_{Q \in \QQ_n(x,r)} \mu(Q)^q}{\log \gamma_n}\\
  &=\lim_{r\downarrow 0}\liminf_{n\to\infty}\frac{-q\sum_{j=1}^n\log N_j +\left(\log\left(\frac{\#\QQ_n(x,r)}{\#\QQ_n}\right)+\sum_{j=1}^n\log N_j\right)}{\sum_{j=1}^n\log c_j}\\
  &=\liminf_{n\to\infty}\frac{(1-q)\log\sum_{j=1}^n\log N_j}{\sum_{j=1}^n\log c_j}\,.
\end{align*}
Here we have used the definition of $\mu$, the fact that $\#\QQ_n(x,r)/\#\QQ_n>c>0$ for some $c>0$ independent of $n$ (but allowed to depend on $x$ and $r$), and the identity
$\#\QQ_n = \prod_{j=1}^n N_j$. Since $q>1$ we have
$$\dim_q(\mu)=\frac{\tau_q(\mu)}{q-1}=\liminf_{n\to\infty}\frac{\sum_{j=1}^n\log N_j}{-\sum_{j=1}^n\log c_j}$$
finishing the proof of \eqref{qge1}.

If $0<q<1$, then the above calculation gives
\begin{equation*}
  \dim_q(\mu,x)=\dim_q(\mu)=\limsup_{n\to\infty}\frac{\sum_{j=1}^n\log N_j}{-\sum_{j=1}^n\log c_j}.
\end{equation*}
for all $x\in E$.
\end{proof}

\bibliographystyle{abbrv}
\bibliography{moran}

\end{document}